\documentclass[11pt]{amsart}
\usepackage{amssymb,amsmath,amsthm}
\usepackage[all,cmtip]{xy}
\usepackage{graphicx}
\usepackage{fullpage,hyperref,enumerate}
\newtheorem{thm}[equation]{Theorem}
\newtheorem{lem}[equation]{Lemma}
\newtheorem{thmA}{Theorem}

\newtheorem{prop}[equation]{Proposition}
\newtheorem{cor}[equation]{Corollary}

\newtheorem*{thm*}{Theorem}
\theoremstyle{definition}

\newtheorem*{rmks}{Remarks}

\newtheorem*{question}{Question}
\numberwithin{equation}{section}

\DeclareMathOperator{\Lk}{Lk}

\DeclareMathOperator{\supp}{supp}
\DeclareMathOperator{\CAT}{CAT}
\DeclareMathOperator{\id}{id}

\DeclareMathOperator{\Mass}{Mass}
\DeclareMathOperator{\FV}{FV}
\DeclareMathOperator{\FVol}{FVol}

\DeclareMathOperator{\Dist}{Dist}
\DeclareMathOperator{\lift}{lift}

\newcommand{\F}{\mathcal{F}}

\begin{document}
\title{Homological filling functions with coefficients}
\author{Xingzhe Li}
\address{Department of Mathematics, University of California, Santa Barbara, California, USA}
\email{xingzheli@ucsb.edu}
\author{Fedor Manin}
\email{manin@math.ucsb.edu}

\begin{abstract}
    How hard is it to fill a loop in a Cayley graph with an unoriented surface?  Following a comment of Gromov in ``Asymptotic invariants of infinite groups'', we define homological filling functions of groups with coefficients in a group $R$. Our main theorem is that the coefficients make a difference.  That is, for every $n \geq 1$ and every pair of coefficient groups $A, B \in \{\mathbb{Z},\mathbb{Q}\} \cup \{\mathbb{Z}/p\mathbb{Z} : p\text{ prime}\}$, there is a group whose filling functions for $n$-cycles with coefficients in $A$ and $B$ have different asymptotic behavior.
\end{abstract}
\maketitle
\section{Introduction}
Geometric group theorists have studied a wide variety of isoperimetric phenomena.  The best-known of these is the Dehn function, which admits a combinatorial algebraic interpretation: it measures the complexity of the word problem, that is, the number of relators needed to trivialize a word of a certain length.  However, once we move into the realm of pure geometry, the ordinary Dehn function, which measures the difficulty of filling a loop in the Cayley complex with a disk, is no more natural than the homological Dehn function, which measures the difficulty of filling a loop with an oriented surface of any genus.  In fact, in some ways the latter is easier to work with.

Both these ideas admit higher-dimensional generalizations.  Given a group $G$ of type $\mathcal{F}^{n+1}$, let $X$ be a finite $n$-connected $(n+1)$-complex on which $G$ acts geometrically.  Then the difficulty of filling an $n$-sphere in $X$ with an $(n+1)$-disk \cite{AWP} or an $n$-cycle with an $(n+1)$-chain \cite[Ch.\ 10]{EPC} are functions which (up to an asymptotic notion of equivalence which we denote $\sim$) depend only on $G$.  This defines filling functions typically denoted $\delta^n_G$ and $\FV^{n+1}_G$, respectively.\footnote{We are stuck with the difference in indices for historical reasons.}

A natural question is whether these homotopical and homological filling functions are always equivalent to each other.  This is known: they are inequivalent for $n=1$ \cite{ABDY} and $n=2$ \cite{Young} but equivalent for $n \geq 3$ (as is shown in \cite{ABDY} by combining results from \cite{BBFS,GroFRM,White}).

Oddly, no one seems to have asked about filling loops with unoriented surfaces.  But once one is working with chains, it is natural (as already remarked by Gromov in \cite[p.~81]{GroAI}) to try to vary the coefficients.  In \cite{Ger}, Gersten defined a filling function $\FV_{\mathbb{R}}$ which measures the difficulty of filling an integral 1-cycle with a \emph{real} (or, equivalently, rational) 2-chain, and asked whether this is asymptotically equivalent to the integral version.  This rational filling function was also studied by Mart\'\i nez-Pedroza \cite{M-P}; we know of no other similar work.

In this paper, given an abelian group $R$, we will denote by $\FV^{n+1}_{G;R}$ the difficulty of filling an $n$-cycle (with coefficients in $\mathbb{Z}$ or $R$, depending on $R$) with an $(n+1)$-chain with coefficients in $R$.  Our main result is that, unlike $\delta^n_G$ and $\FV^{n+1}_G$, these differ in all dimensions:
\begin{thmA} \label{thm:main}
    Let $q$ be a prime, $n \geq 1$, and $d \in \mathbb{N} \cup \{\infty\}$.  Then there is a group $H$ of type $\mathcal{F}^{n+1}$ (a normal subgroup of an $(n+2)$-dimensional CAT(0) group) such that
    \begin{align*}
        \FV^{n+1}_{H;R}(x) &\preceq x^{(n + 1)[\ln(n + 1) + 2]} & R
        &= \mathbb{Q}\text{ or }\mathbb{Z}/p\mathbb{Z}, \gcd(p,q)=1 \\
        \FV^{n+1}_{H;R}(x) &\succeq f_{d,n}(x)
        & R &= \mathbb{Z}\text{ or }\mathbb{Z}/q\mathbb{Z},
    \end{align*}
    where $f_{d,n}(x)=\exp(\!\sqrt[n]{x})$ if $d=\infty$ and $x^{d/n}$ otherwise.
\end{thmA}

\subsection{Why filling functions with coefficients?}
Of the filling functions introduced in this paper, the easiest to understand and perhaps the most useful are $\FV^n_{G;\mathbb{Q}}$ and $\FV^n_{G;\mathbb{Z}/2\mathbb{Z}}$.  A rational filling of a cycle $Z$ is one in which simplices are allowed to appear with fractional coefficients.  One can also think of the rational filling volume as measuring the difficulty of filling multiples of a cycle:
\[\FVol_{\mathbb{Q}}(Z)=\liminf_{r \to \infty} \frac{\FVol(rZ)}{r}.\]
Rational chains and cycles can be useful mainly because, as elements of a vector space, they simplify algebraic arguments.

On the other hand, a mod $2$ filling of a cycle can be thought of as a filling by an unoriented hypersurface.  In fact, mod $2$ homology is easier to define than integral homology as one does not have to worry about orientation.  As such, and because $k$-chains can be thought of as subsets of the set of $k$-cells, it is widely used in combinatorics and applied topology.  Mod $2$ filling functions, specifically, have come up in the study of high-dimensional expanders.  While there are many nonequivalent candidate higher-dimensional generalizations of the notion of expander graphs \cite{Lub}, several of them, e.g.~\cite{LiMe,Gro,DK}, use a coisoperimetric constant which measures the difficulty of filling cocycles.  Given the close relationship between isoperimetry and coisoperimetry induced by linear programming duality, 
and the fact that many explicit constructions of expander families, from \cite{Marg} and \cite{LPS} to \cite{EvKa}, have used geometric group theory, we are optimistic that the ideas discussed here are relevant in that domain.

Mod $2$ isoperimetry in the Euclidean setting was notably studied by Robert Young \cite{YoungMod2}, who showed that the mod $2$ filling volume of a Lipschitz cycle is bounded by a constant times its integral filling volume.  In other words, it is impossible to build a cycle in Euclidean space (even one that is not at all isoperimetric) which is much easier to fill mod $2$ (or, more generally, mod $p$) than integrally.  Among other implications, this considerably simplifies the geometric measure theory of chains with mod 2 coefficients.
\begin{question}
Can an infinitesimal version of our construction yield exotic metric spaces which are locally highly connected but in which Young's results do not hold?
\end{question}

\subsection{Proof methods}

In the case $n=1$, the proof closely follows the methods of \cite{ABDY}.  Essentially, the group $H$ is constructed by amalgamating a group with large Dehn function with a Bestvina--Brady group which has many cycles that can be filled mod $p$, but not integrally.  The resulting group has many cycles that have a small mod $p$ filling built by taking ``shortcuts'' through the Bestvina--Brady group, but only a very large integral filling.

To extend to higher dimensions, we use a ``suspension'' construction which is somewhat similar to that used in \cite{BBFS} to build groups with prescribed higher-order Dehn functions.  In our case, we use the fact that $H$ is the kernel of a homomorphism $G \to \mathbb{Z}$, where $G$ is a CAT(0) group.  Given such a homomorphism, the group $G *_H G$ is the kernel of a homomorphism $G \times F_2 \to \mathbb{Z}$.  (For example, if $G$ is a RAAG and $H$ is its Bestvina--Brady group, then $G \times F_2$ and $G *_H G$ are the RAAG and Bestvina--Brady group whose associated flag complex is the simplicial suspension of that of $G$ and $H$.)  To prove Theorem \ref{thm:main}, we iterate this construction and apply the following result:
\begin{thmA}\label{B} 
    Let $n \geq 2$ and let $R$ be a quotient of $\mathbb{Z}$.  Then for any group $G$ of type $\mathcal{F}^{n+1}$ satisfying $\FV^{n}_{G; R}(x) \preceq \exp(x)$, and subgroup $H$ of type $\mathcal{F}^{n}$,
    \[\FV^n_{H;R}\bigl({\overline{\FV^n_{G;R}}\,}^{-1}(x)\bigr) \preceq \FV^{n+1}_{G *_H G;R}(x) \preceq x\overline{\FV^{n+1}_{G;R}}\bigl(\overline{\FV^n_{H;R}}(x)\bigr),\]
    where we write $\overline f(x)=\max(f(x),x)$ (noting that $\overline f \sim f$).
\end{thmA}
The rational case is not identical, but substantially similar.

\subsection*{Acknowledgements}
The second author would like to thank Robert Young for suggesting looking at the paper \cite{ABDY}.  We would also like to thank the anonymous referee for a close reading leading to a large number of corrections and helpful suggestions, and in particular for noticing a missing assumption in Theorem \ref{B}.  Both authors were partially supported by NSF individual grant DMS-2001042.

\section{Preliminaries}

\subsection{Definition of homological filling functions with coefficients}
Recall that a group $G$ is of type $\F^n$ if there is a $K(G,1)$ with finite $n$-skeleton, or equivalently, if it acts freely and geometrically on an $(n-1)$-connected cell complex (for example, the universal cover of the $n$-skeleton of this $K(G,1)$).  A group is of type $\F^1$ if and only if it is finitely generated and $\F^2$ if and only if it is finitely presented.

As customary in geometric group theory, we use a relation $\preceq$ defined on functions $\mathbb {R}^{\geq 0} \rightarrow \mathbb{R}^{\geq 0}$ to capture inequality of growth rates.  We write $f \preceq g$ if there exists a $C > 0$ such that
\[\text{for every }x \geq 0,\quad f(x) \leq Cg(Cx + C) + Cx + C.\]
We say that $f \sim g$ whenever $f \preceq g$ and $g \preceq f$.

Let $G$ be a group of type $\F^{n + 1}$, and let $X$ be an $n$-connected complex on which it acts freely and geometrically.  First, let $R=\mathbb{Z}$ or $\mathbb{Q}$. If $\alpha$ is a cellular $n$-cycle in $X$, we let
\[\FVol^{n + 1}_{X; R}(\alpha) = \inf\{\Mass^{n + 1} (\beta) \:|\: \beta \in C_{n + 1}(X; R), \partial \beta = \alpha\},  
\]
where $\Mass^{n+1} (\beta) := \lVert\beta\rVert_{1} = \sum |b_{i}|$ provided that $\beta = \sum b_{i}\Delta_{i}$ is a sum of $(n+1)$-cells of $X$ with $b_{i} \in R$. We define the $n$-dimensional homological filling function of $X$ with coefficients in $R$ to be
\[\FV^{n + 1}_{X; R}(x) = \sup \bigl\{\FVol^{n + 1}_{X; R}(\alpha) \mid \alpha \in Z_{n}(X; \mathbb{Z}), \Mass^{n} (\alpha) \leq x\bigr\}.\]
By \cite[Lemma 1]{Young}, the growth rate of this function (up to the relation $\sim$) depends only on the quasi-isometry type of $G$, so we can write
\[\FV^{n + 1}_{G; R}(x) := \FV^{n + 1}_{X; R}(x).\]
When $n=1$, this is known as the \emph{homological Dehn function}.

Now suppose that $R$ is a finite abelian group.  We define the function $\FV^{n+1}_{X;R}$ in almost the same way, with two differences: (1) the mass of a chain is defined simply to be the number of cells in its support and (2) we define
\[\FV^{n + 1}_{X; R}(x) = \sup \bigl\{\FVol^{n + 1}_{X; R}(\alpha) \mid \alpha \in Z_{n}(X; R), \Mass^{n} (\alpha) \leq x\bigr\}.\]
That is, we maximize the filling volume over cycles with coefficients in $R$ rather than integral cycles.

It is clear that the rational filling function must use fillings of integral cycles to be interesting: otherwise one could scale any cycle until its mass is less than some threshold, and the function would always be either linear or infinite.  For $\mathbb{Z}/p\mathbb{Z}$, which is our other major example, we could fill either mod $p$ cycles or reductions of integral cycles; we use mod $p$ cycles because it works better with some of our proofs, but it's not clear whether this subtle difference in definitions makes a difference in this case.

It certainly doesn't make a difference when $n=1$, since in that case every mod $p$ cycle is the image of an integral cycle which is not too much larger:
\begin{prop} \label{lift}
    Suppose that $X$ is a graph and $p$ is any integer.  Then every cellular cycle $\alpha \in Z_1(X;\mathbb{Z}/p\mathbb{Z})$ has a preimage $\tilde \alpha \in Z_1(X;\mathbb{Z})$ such that $\Mass(\tilde \alpha) \leq p\Mass(\alpha)$.
\end{prop}
\begin{proof}
  We construct $\tilde\alpha$ explicitly as follows.  Take a chain $\eta \in C_1(X;\mathbb{Z})$ which lifts $\alpha$ and has coefficients between $-p/2$ and $p/2$.  (This is unique if $p$ is odd, but may involve choices for even $p$.)  Then $\partial\eta$ is a $0$-chain with coefficients in $p\mathbb{Z}$.  We take $\tilde\alpha=\eta-p\theta$ where $\theta$ is a minimal filling of $\frac{1}{p}\partial\eta$.  To prove the lemma, it is enough to show that $\Mass(\theta) \leq \frac{1}{p}\Mass(\eta)$.
  
  To see this, notice that $\theta$ is the most efficient ``matching'' via geodesics between the positive and negative points in $\frac{1}{p}\partial\eta$, with multiplicity.  This kind of minimal matching problem always has an integer solution that is optimal even among real solutions, see e.g.~\cite[\S3.2]{MG}.   Therefore $p\theta$ is a minimal filling of $\partial\eta$.  Since $\eta$ is also a filling of $\partial\eta$, we know that $\Mass(\eta) \geq \Mass(p\theta)$.
\end{proof}
In Section 3, we will use the fact that $\FV^2_{G;\mathbb{Z}/p\mathbb{Z}}$ can be defined equivalently by maximizing over all mod $p$ cycles or only over images of integral cycles.

\subsection{Right-angled Artin groups}
The construction of the groups $G$ and $H$ in Theorem \ref{thm:main} uses some facts about right-angled Artin groups.  Here we give the definition and some of their properties.

Given a simplicial graph $\Lambda$ with vertex set $V(\Lambda)$ and edge set $E(\Lambda)$, the associated \emph{right-angled Artin group} (RAAG) has the presentation 
\[
A := \langle V(\Lambda)\:|\:[i(e), t(e)] = 1, \forall e \in E(\Lambda) \rangle, 
\]
where $i(e)$ and $t(e)$ are the endpoints of $e$.  There exists a $K(A,1)$, the \emph{Salvetti complex} $X_{A}$ of $A$, which is a one-vertex locally $\CAT(0)$ cube complex. Let $h_{A}: A \rightarrow \mathbb{Z}$ be the group homomorphism sending each generator of $A$ to $1$; there is a cube-wise linear map $h_{X_{A}}: X_{A} \rightarrow S^{1}$ which induces $h_{A}$ on fundamental groups. The lift of $h_{X_{A}}$ to the universal cover gives us a $h_{A}$-equivariant Morse function $h_{X_{A}}: \widetilde{X_{A}} \rightarrow \mathbb{R}$.  With respect to this Morse function, one defines the \emph{ascending} and \emph{descending links} of a vertex $v$ as the subcomplexes of its link which are, respectively, above and below $v$.  Since $X_{A}$ has only one vertex, we denote them by $\Lk_{\uparrow}(\widetilde{X_{A}})$ and $\Lk_{\downarrow}(\widetilde{X_{A}})$ respectively.  The following theorem describes the topology of the ascending and descending links as well as the level set $L_{A} := h_{X_{A}}^{-1}(0)$. 

\begin{thm}[Bestvina and Brady \cite{BB}] \label{BBthm}
If $\Lambda$ is the $1$-skeleton of the flag simplicial complex $Y$, then both $\Lk_{\uparrow}(\widetilde{X_{A}})$ and $\Lk_{\downarrow}(\widetilde{X_{A}})$ are isomorphic to $Y$. Moreover, with $h_{A}$, $h_{X_{A}}$ the maps defined above, $H_{A} = \ker h_{A}$ acts on the complex $L_{A} = h_{X_A}^{-1}(0)$, which is homotopy equivalent to a wedge product of infinitely many copies of $Y$, indexed by the vertices in $\widetilde{X_{A}} \setminus L_{A}$.  In fact, $L_A$ is a union of scaled copies of $Y$.
\end{thm}

The proof of this is based on discrete Morse theory, see \cite[Theorems 5.12 and 8.6]{BB}. 

\subsection{Two general lemmas}
We first state a $\CAT(0)$ isoperimetric inequality due to Wenger \cite{Wenger}; cf.~\cite[Proposition~2.4]{ABDDY}.
\begin{prop} \label{CAT} 
If $X$ is a $\CAT(0)$ polyhedral complex and $n \geq 1$, then the $n$-dimensional homological filling function of $X$ with coefficients in $R$ satisfies 
\[\FV^{n + 1}_{X; R}(x) \preceq x^{\frac{n + 1}{n}},\]
where $R = \mathbb{Z}$, $\mathbb{Q}$, or $\mathbb{Z}/p\mathbb{Z}$, $p \geq 2$.  Moreover, there is a constant $c$ such that if $\alpha \in Z_{n}(X; R)$ (or $\alpha \in Z_n(X;\mathbb{Z})$ if $R=\mathbb{Q}$), then there exists a chain $\beta \in C_{n + 1}(X; R)$ such that $\partial \beta = \alpha$,
\[\Mass^{n + 1} (\beta) \leq c [\Mass^{n} (\alpha)]^{\frac{n + 1}{n}}, 
\]
and $\supp \beta$ is contained in a $c [\Mass^{n} (\alpha)]^{\frac{1}{n}}$-neighborhood of $\supp \alpha$. 
\end{prop}
\begin{proof}
For $R = \mathbb{Z}$ or $\mathbb{Z}/p\mathbb{Z}$, this follows from the proof of \cite{Wenger}.  While it is stated only for $\mathbb{Z}$ and $\mathbb{Z}/2\mathbb{Z}$, the proof works equally for all $\mathbb{Z}/p\mathbb{Z}$ once the notion of mass is properly defined.  The result for $\mathbb{Z}$ implies that for $\mathbb{Q}$ by definition.
\end{proof}

Another general lemma equates two ways of constructing groups; see \cite{Bau97} for some related results.
\begin{prop} \label{amalg}
Suppose that we have a commutative diagram with exact rows,
\[\xymatrix{1 \ar[r] & H \ar[r] \ar@{^(->}[d] & G \ar[r]^-{h} \ar@{^(->}[d]^{(\id, 1)} & \mathbb{Z} \ar[r] \ar@{=}[d] & 1 \\
    1 \ar[r] & H' \ar[r] & G \times F_{2}(u, v) \ar[r]^-{h'}& \mathbb{Z} \ar[r] & 1,}\]
where $h'$ sends $u$ and $v$ to $1$. Suppose furthermore that $G$ has a generating set such that $h$ maps every generator to $0$ or $\pm{1}$.  Then $H' \cong G *_{H} G$.
\end{prop}
\begin{proof}
    Given $g \in G$, write $g_\ell$ and $g_r$ for the elements of $G *_H G$ corresponding to ``$g$ on the left'' and ``$g$ on the right''.  Define a map $G *_H G \to H'$ which, for any $g \in G$, sends
    \[g_\ell \mapsto gu^{-h(g)}\qquad\text{and}\qquad g_r \mapsto gv^{-h(g)}.\]
    The reader can confirm that this extends to an isomorphism.  In particular, injectivity follows from \cite[Theorem 1]{Bau97}.
\end{proof}

\section{The case $n=1$}
In this section, we prove Theorem \ref{thm:main} in the case $n=1$.  This construction fairly closely follows that of \cite{ABDY}.  In what follows, let $q$ be a prime and $d$ a positive integer or infinity.

\subsection{Constructing the groups $G$ and $H$}
We will build $G$ as the fundamental group of a graph of groups whose vertices are labeled with two CAT(0) groups: a RAAG $A_Y$ corresponding to a flag simplicial complex $Y$, and a group $Q$ which we import from \cite{ABDY} along with its desired properties.  Each edge will be labeled with the group $E = F_2 \times F_2$.

Let $K_{q}$ be a CW complex consisting of $S^{1}$ and a single $2$-cell glued on via an attaching map of degree $q$. Then
\[H_{1}(K_{q}; R) \cong 
\begin{cases}
\mathbb{Z}/q\mathbb{Z} & R=\mathbb{Z}\text{ or }\mathbb{Z}/q\mathbb{Z}\\
0 & R=\mathbb{Q}\text{ or }\mathbb{Z}/p\mathbb{Z}, \gcd(p,q)=1.
\end{cases}
\]
We equip this complex with a flag triangulation in which the $1$-cell is subdivided into four edges between four vertices $a$, $u_{1}$, $s$, $v_{1}$ in order. Label the remaining vertices $y_{1}, y_{2}, \ldots, y_\ell$.  Call the resulting simplicial complex $Y$.

By definition, $Y$ is a connected, $2$-dimensional finite flag complex with $H_{1}(Y; R) = 0$ whenever $R = \mathbb{Q}$ or $\mathbb{Z}/p\mathbb{Z}$ and $\gcd(p,q)=1$.  Moreover:
\begin{prop} \label{prop:1/q}
Any integral $1$-cycle in $Y$ has a rational filling with coefficients in $(1/q)\cdot\mathbb{Z}$.
\end{prop}
\begin{proof}
Write $p:\widetilde{Y} \to Y$ for the universal covering map.  Define a map $\lift:C_n(Y) \to C_n(\widetilde{Y})$ which sends a cell to the sum of its preimages in the universal cover.  This is a chain map, so given a cycle $\gamma \in Z_1(Y)$, $\lift(\gamma)$ is a cycle in $Z_1(\widetilde{Y})$ with an integral filling $\beta$.  Then $\frac{1}{q}p_\#\beta$ is a filling of $\gamma$ with coefficients in $(1/q)\cdot\mathbb{Z}$.
\end{proof}
Integrally and mod $q$, the edges connecting the vertices $a$, $u_{1}$, $s$, $v_{1}$ form a homologically nontrivial cycle. Let $\Lambda$ be the $1$-skeleton of $Y$, and denote the associated RAAG by $A$. Recall that $h_{A}: A \rightarrow \mathbb{Z}$ is a group homomorphism sending each generator of $A$ to $1 \in \mathbb{Z}$. 

We take $Q := B \times F_{2}(u_{1}, v_{1})$, where $B$ is a group defined in \cite{ABDY} which satisfies the following properties:
\begin{enumerate}[(i)]
\item $B$ has a presentation with generators $a_{1}, a_{2}, \ldots, a_{m - 1}, s,$ and $t$ such that the only relation involving $s$ is $[s,t]=1$.  Moreover, the corresponding presentation complex is a $2$-dimensional locally CAT(0) cube complex $X_B$, which is a $K(B,1)$.
\item The Cayley graph of the (free) subgroup generated by $a_{i}$ and $s$ is convexly embedded in $\widetilde{X_{B}}$.
\item Let $h_{B}: B \rightarrow \mathbb{Z}$ be the group homomorphism sending $a_{i}$, $s$, $t$ to $1 \in \mathbb{Z}$, and $h_{X_{B}}: \widetilde{X_{B}} \rightarrow \mathbb{R}$ be the Morse function extended from $h_{B}$.  Then $\ker h_B$ is a free group; equivalently, the level set $L_{B}:= h_{X_{B}}^{-1}(0)$ is a tree.
\item $B$ is isomorphic to a free-by-cyclic group $F_m \rtimes_{\phi} \mathbb{Z}$, where $F_m=\ker h_B$.
\item The ascending and descending links $\Lk_{\uparrow}(X_{B})$ and $\Lk_{\downarrow}(X_{B})$ are trees.
\item The distortion of the subgroup $F_{m}=\ker h_B \subset B$ is sufficiently large:
\[\Dist_{F_m}(x) := \max\{|g|_{F_{m}}: g \in F_m, |g|_{B} \leq x\} \sim f_d(x),\]
where (as in the statement of Theorem \ref{thm:main}) $f_d(x)=e^x$ if $d=\infty$ and $x^d$ otherwise.  (We use $|w|_\Gamma$ to represent wordlength in a group $\Gamma$.)
\end{enumerate}

Note that (i) implies that $X_{Q} := X_{B} \times (S^{1} \vee S^{1})$ is a 3-dimensional $\CAT(0)$ cube complex, and hence a $K(Q, 1)$.

Let $h_{Q}: Q \rightarrow \mathbb{Z}$ be the group homomorphism sending generators $a_{i}, s, t, u_{1}, v_{1}$ to $1$.  Then by Proposition \ref{amalg}, $\ker(h_Q) \subset Q$ is isomorphic to $D :=B *_{F_m} B$, the Bieri--Stallings double of $B$.

For each $1 \leq i \leq m - 1$, the subgroup of $B$ generated by $a_{i}$ and $s$ is free, so $Q$ contains subgroups of the form $E_{i} = F_{2}(a_{i}, s) \times F_{2}(u_{1}, v_{1})$ for each $i$. On the other hand, $A$ contains the subgroup $E = F_{2}(a, s) \times F_{2}(u_{1}, v_{1})$.  Identifying each $E_i \subset Q$ with $E$ in a copy $A_i$ of $A$ gives a graph of groups with vertices $Q,  A_{1}, \cdots, A_{m - 1}$ and edges $E_{1}, \cdots, E_{m - 1}$.
Let $G$ be the fundamental group of this graph of groups. By construction, we can express $G$ and its subgroups as 
\begin{gather*}
G = \langle a_{1}, \ldots, a_{m - 1}, s, t, u_{1}, v_{1}, y_{j}^{i}\rangle \qquad i = 1, \ldots, m - 1, j = 1, \ldots, \ell;\\
Q = B \times F_{2}(u_{1}, v_{1}) = \langle a_{1}, \ldots, a_{m - 1}, s, t, u_{1}, v_{1}\rangle;\\
A_{i} = \langle a_{i}, s, u_{1}, v_{1}, y_{1}^{i}, \ldots, y_\ell^{i}\rangle;\\
E_{i} = F_{2}(a_{i}, s) \times F_{2}(u_{1}, v_{1}). 
\end{gather*} 
We extend the group homomorphisms $h_{A}: A \rightarrow \mathbb{Z}$ and $h_{Q}: Q \rightarrow \mathbb{Z}$ to $h: G \rightarrow \mathbb{Z}$. Let $H = \ker h$. 

\subsection{Properties of $G$ and $H$}
\begin{lem} \label{n=1:CAT(0)}
$G$ is a $\CAT(0)$ group.
\end{lem}
\begin{proof}
Let $X_{E_i} := (S^{1} \vee S^{1}) \times (S^{1} \vee S^{1})$, along which $X_{A_i}$ and $X_{Q}$ are glued together. Since $X_{E_i}$ is a $K(E_i, 1)$ and is convex as a subset of $X_{A_i}$ and $X_Q$, attaching $m - 1$ copies of $X_{A_i}$ to $X_{Q}$ along the $X_{E_i}$ gives us a locally $\CAT(0)$ cube complex, called $X_{G}$. Since $G$ acts cocompactly on $\widetilde{X_{G}}$ by deck transformations, it is $\CAT(0)$. 
\end{proof}

Recall that $H = \ker h$.  Using the work of Bestvina and Brady \cite{BB}, $h$ extends to an $h$-equivariant Morse function $\widetilde{X_{G}} \rightarrow \mathbb{R}$, which we also call $h$ by an abuse of notation. If we set $L_{G} = h^{-1}(0)$, then $H$ acts cocompactly on the level set $L_{G}$.  To show $H$ is finitely presented, it's enough to show that $L_G$ is simply connected.

\begin{lem}\label{FinPro}
$H$ is a finitely presented group.
\end{lem}
\begin{proof}
By \cite[Theorem 4.1]{BB}, it suffices to prove that the ascending and descending links of any vertex in $\widetilde{X_{G}}$ are simply-connected. Since $X_{G}$ has only one vertex, all links are copies of a single complex $\Lk(X_{G})$.

The vertices of the ascending link $\Lk_\uparrow(X_G)$ correspond to generators of $G$; we denote the vertex corresponding to a generator $x$ by $x^+$.  The complex $\Lk_{\uparrow}(X_{G})$ is formed by gluing $\Lk_{\uparrow}(X_{Q})$ and $\Lk_{\uparrow}(X_{A_{i}})$'s together along the subcomplexes $S_{i}$ spanned by $a_{i}^{+}$, $u_{1}^{+}$, $s^{+}$, and $v_{1}^{+}$. Consider $\Lk_{\uparrow}(X_{Q})$ first. The fact that $\Lk_{\uparrow}(X_{Q})$ is the suspension of a tree (with suspension points $u_{1}^{+}$ and $v_{1}^{+}$) implies that $\pi_{1}(\Lk_{\uparrow}(X_{Q})) \cong 0$. By Theorem \ref{BBthm}, we know that $\Lk_{\uparrow}(X_{A_{i}})$ is isomorphic to $Y$, and in particular $\pi_{1}(\Lk_{\uparrow}(X_{A_{i}})) \cong \mathbb{Z}/q\mathbb{Z}$ is generated by a loop that lies in $S_i$.  This means that any loop in $\Lk_{\uparrow}(X_{G})$ is homotopic to one that lies in $\Lk_{\uparrow}(X_Q)$.  Therefore, $\Lk_{\uparrow}(X_{G})$ is simply-connected.

Similarly, $\Lk_{\downarrow}(X_{G})$ must be simply-connected. It follows that $L_{G}$ is simply-connected and $H$ is finitely presented.
\end{proof}

\subsection{Upper bound on homological Dehn
functions}
Recall that the level set $L_{G} := h^{-1}(0)$. Let us define $L_{Q} := L_{G} \cap \widetilde{X_{Q}}$, $L_{A_{i}} := L_{G} \cap \widetilde{X_{A_{i}}}$, and $L_{E_{i}} = L_{G} \cap \widetilde{X_{E_{i}}}$. Since the index $i$ represents multiple copies of the same object, we sometimes drop it from the notation.

Topologically, both $L_{A}$ and $L_{E}$ are non--simply connected. By Theorem \ref{BBthm}, $L_{A}$ is homotopy equivalent to a wedge sum of scaled copies of $Y$ while $L_{E}$ is homotopy equivalent to a wedge sum of scaled copies of the square. Although $\pi_{1}(L_{A})$ is non-trivial, we have $H_{1}(L_{A}; R) = 0$ when $R = \mathbb{Q}$ or $\mathbb{Z}/p\mathbb{Z}$ and $\gcd(p,q)=1$, which enables us to fill any cycle in $L_{A}$ with coefficients in $\mathbb{Q}$ or $\mathbb{Z}/p\mathbb{Z}$. Using the methods of \cite{ABDDY,ABDY}, we can further show that the growth rate of filling functions with these coefficients is at most $x^5$.  As we shall see, this is significantly different from the growth of $\FV^{2}_{H; \mathbb{Z}}$ or $\FV^{2}_{H; \mathbb{Z}/q\mathbb{Z}}$.

\begin{prop} \label{UpperBound}
With $H$ the kernel group defined above, we have $\FV^{2}_{H; R}(x) \preceq x^{5}$ whenever $R = \mathbb{Q}$ or $\mathbb{Z}/p\mathbb{Z}$, $\gcd(p,q) = 1$.  Moreover, every integral $1$-cycle has a rational filling of mass $\preceq x^5$ whose coefficients lie in $(1/q)\cdot\mathbb{Z}$.
\end{prop}
\begin{proof} 
We closely follow the proof of \cite[Proposition 6.1]{ABDY}, with some alterations.  Thus we describe the construction in a relatively informal way, relying on \cite{ABDY} for some details.

By Proposition \ref{lift}, it suffices to show that the image of any integral cycle $\alpha \in C_1(L_{G})$ of mass at most $x$ can be filled by a chain in $C_2(L_{G}; R)$ of mass $\preceq x^5$.  Since $x^5$ is a superadditive function, it is enough to show this when $\alpha$ is a loop of length $x$.

By construction, $L_{G}$ consists of copies of $L_{A}$ and $L_{Q}$ glued together along copies of $L_{E}$. So, we can first construct a filling when $\alpha$ is supported in one copy of $L_{A}$ or $L_{Q}$, and then extend to the case when $\alpha$ travels through multiple copies of $L_{A}$ and $L_{Q}$. 

Consider the filling of a loop $\alpha$ in $L_{A}$. By Proposition \ref{CAT}, we obtain an integral $2$-chain $\beta \in C_2(\widetilde{X_A};\mathbb{Z})$ which fills $\alpha$ such that $\Mass \beta \preceq x^{2}$ and $\beta$ is supported in $h_{A}^{-1}([-cx, cx])$, where $c$ is a universal constant.  We turn this into a filling in $L_A$ of mass $\preceq x^4$, as follows.  Let
\[Z=\widetilde{X_{A}} \setminus \bigcup_{v \notin L_{A}} B^{\circ}_{1/4}(v)\]
be the space formed by deleting open neighborhoods of vertices of $\widetilde{X_{A}}$ outside $L_{A}$.  According to \cite[Theorem 4.2]{ABDDY}, the resulting Swiss cheese retracts to $L_A$ via a map $\rho:Z \to L_A$.  Moreover, $\rho$ is $(cx)$-Lipschitz on $h_{A}^{-1}([-cx, cx])$. Therefore, it increases area by a factor of $\preceq x^2$, meaning that $\Mass(\rho(Z \cap \beta)) \preceq x^4$.

Now, the boundaries of the holes in $Z \cap \beta$ form a chain $\gamma=\partial(Z \cap \beta)-\alpha$ whose total length is $\preceq x^2$.  This chain lives in a union of copies of $Y$ in which it is nullhomologous.  Since $Y$ is a compact space with finite fundamental group, homological fillings of $1$-cycles in $Y$ have area linear in the length of the cycle.

The image under the retraction, $\rho(\gamma)$, lives in a union of copies of $Y$ each scaled by at most $cx$, and therefore has a filling with coefficients in $R$ of mass $\preceq x^4$; when $R=\mathbb{Q}$, this has coefficients in $(1/q)\cdot\mathbb{Z}$ by Proposition \ref{prop:1/q}.  Together with $\rho(Z \cap \beta)$, this gives a filling of $\alpha$ of mass $\preceq x^4$.

Now we consider $\alpha$ lying in $L_Q$. We use the same groups $Q$ and $E$ as in \cite{ABDY}, so their construction of fillings in $L_Q$ carries over in a straightforward way.  They show that every 1-cycle in $L_Q$ has a filling of mass $\preceq x^4$ if one takes advantage of ``shortcuts'' through $L_A$.  More precisely, recall that $\pi_1(L_{E_i})$, for each $i$, is infinitely generated by scaled copies of the square.  Any trivial word in $D$ of length $x$ is homologous in $L_Q$, through an integral chain of area $\preceq x^4$, to a sum of $\preceq x^2$ of these generators, each at scale $\preceq x$.  Each of these generators has a homological filling in $L_A$ of mass $\preceq x^2$ (in our case, with coefficients in $(1/q)\cdot\mathbb{Z} \subset \mathbb{Q}$ or in $\mathbb{Z}/p\mathbb{Z}$).

Now, let $\alpha$ be a loop which travels through multiple copies of $L_{A}$ and $L_{Q}$; in this case we again use the proof from \cite{ABDY} verbatim.  To make this situation easier to analyze, we build a new complex $L$ on which $H$ acts by ``stretching'' each copy of $L_{E}$ in $L_{G}$ into a product $L_{E} \times [0, 1]$. Since $H$ acts geometrically on $L_{G}$, it acts geometrically on the new level set $L$ as well.

The thrust of the argument is to inductively ``lop off'' subpaths of $\alpha$ living in either $L_A$ or $L_Q$, without increasing the overall length of the path.  Based on the normal form theorem for graphs of groups, there is a subpath of the form $t \gamma t'$ that enters a copy of $L_{A}$ or $L_{Q}$ through a copy of $L_{E}$, then leaves through the same $L_{E}$.  Let $\gamma'$ be a geodesic path which connects the endpoint of $t$ with the endpoint of $t'$.  Since $L_E$ is undistorted in $L$, $\gamma'$ is contained in $L_E$.  Thus $\theta=t \gamma t' (\gamma')^{-1}$ is a loop in the union of a copy of $L_{E} \times [0, 1]$ and a copy of $L_{Q}$ or $L_{A}$. This loop is filled by a $2$-chain $\beta$ with coefficients in $R$ (and, in the rational case, in $(1/q) \cdot \mathbb{Z}$) such that 
\[\Mass \beta \preceq (x + l(\gamma'))^{4} + l(\gamma') \preceq x^{4}.
\]
Repeating this process for the loop $\alpha - \theta$ inductively provides us with a filling of $\alpha$. Since there are at most $x$ steps, we conclude that the total filling of $\alpha$ has mass at most $x^{5}$ and integer multiples of $1/q$ as the coefficients. 
\end{proof}

\subsection{Lower bound on homological Dehn functions}
Now, in contrast to the previous subsection, we let the coefficient group $R$ be $\mathbb{Z}$ or $\mathbb{Z}/q\mathbb{Z}$, and prove the following. 
\begin{thm} \label{LowerBound}
$\FV^{2}_{H; R} \succeq \Dist_{F_m}$, where $\Dist_{F_m} \sim f_d$ is the distortion function of the subgroup $F_{m}=h_B^{-1}(0)$ in $B$.
\end{thm}
The overall strategy is as follows.  We first show that there is a loop $\gamma$ of length $2x + 2$ whose filling area in the double $D=B *_{F_m} B$ is $\Dist_{F_m}(x)$; this is a homological version of a theorem of Bridson and Haefliger on Dehn functions \cite[Theorem \textrm{III}.$\Gamma$.6.20]{BH}.  Then we show that this loop is no easier to fill in the larger group $H$.

\begin{lem} \label{lem:BH}
  Let $B$ be any finitely presented group and $F < B$ a finitely generated subgroup with distortion function $\delta$.  Then for any coefficient group $R$,
  \[\FV^2_{B *_F B;R}(x) \succeq \delta(x).\]
\end{lem}
\begin{proof}
We start by building a complex $Y$ whose fundamental group is $B *_F B$.  Let $X_B$ be the presentation complex of $B$, and let $X_F$ be a wedge of circles with each circle labeled by a generator of $F$.  We construct $Y$ by gluing the two ends of $X_F \times [0,1]$ (given the product cell structure) to two copies of $X_B$ via cellular maps which send each circle to the corresponding generator.
  
Now given $x$, we build a loop $\gamma$ in the universal cover $\widetilde Y$ of length $2x+2$ whose filling area is at least $\delta(x)$.  First, choose an element $w \in F$ such that $|w|_B \leq x$ and $|w|_F \geq \delta(x)$.
  
Note that $\widetilde Y$ is glued together from copies of $\widetilde{X_F} \times [0,1]$ and $\widetilde{X_B}$ arranged as the edges and vertices of a tree.  Choose a basepoint $p$ in $\widetilde{X_F}$, and a base copy of $\widetilde{X_F} \times [0,1]$ which we call $Z_0 \times [0,1] \subset \widetilde Y$.  Let $Z_\ell$ and $Z_r$ be the two copies of $\widetilde{X_B}$ on either side of $Z_0$.  Let $\gamma$ start at $(p,0) \in Z_0 \times [0,1]$, follow the geodesic path in $Z_\ell$ to $(w \cdot p,0)$, go across to $(w \cdot p,1)$, follow the geodesic path in $Z_r$ to $(p,1)$, and finally cross back to $(p,0)$.
  
We claim that any cellular filling $\beta \in C_2(\widetilde Y;R)$ of $\gamma$ must have $\Mass \beta \succeq \delta(x)$.  To see this, notice that $\beta \cap \widetilde {X_F} \times \{1/2\}$ is a $1$-chain in $\widetilde{X_F}$ whose boundary is $[w \cdot p] - [p]$.  (Here we translate between the orientation of cells in $\widetilde{X_F} \times [0,1]$ and $\widetilde{X_F} \times \{1/2\}$ using the unit normal pointing towards $1$.)  Since the geodesic distance in $\widetilde{X_F}$ between these two points is at least $\delta(x)$, $\beta \cap \widetilde {X_F} \times \{1/2\}$ is supported on at least $\delta(x)$ edges.  Therefore $\beta$ is supported on at least $\delta(x)$ 2-cells.  This completes the proof when $R=\mathbb{Z}$ or a finite group.

Although we do not need the result for $R=\mathbb{Q}$, we include it for completeness.  In that case, define the cochain $c \in C^0(\widetilde{X_F})$ by $c(v)=d_{\widetilde{X_F}}(p,v)$.  Then $\lVert \delta c \rVert_\infty=1$, and \[\delta c\bigl(\beta \cap \widetilde {X_F} \times \{1/2\}\bigr)=c([w \cdot p]-[p]) \geq \delta(x),\]
which means that $\bigl\lVert \beta \cap \widetilde {X_F} \times \{1/2\} \bigr\rVert_1 \geq \delta(x)$.
\end{proof}

\begin{proof}[Proof of Theorem \ref{LowerBound}]
We choose a hard-to-fill 1-cycle $\gamma$ inside a copy of $L_Q \subset L_G$, by the same method as in the proof of the lemma.  For any $x$, choose a $w \in F$ satisfying $\lvert w \rvert_B \leq x$ and $\lvert w \rvert_F \geq \delta(x)$.  Then build a loop $\gamma$ by taking the geodesic path in one copy of $\widetilde{X_B}$ from $p$ to $w \cdot p$, and then the same geodesic path back in the other copy.

We note some facts about $\gamma$.  First, applying the proof of Lemma \ref{lem:BH} to $D=B *_{F_m} B$, we get a space $\widetilde{Y}$ which maps quasi-isometrically to $L_Q$ by contracting the interval in each copy of $\widetilde{X_F} \times [0,1]$.  This map sends the loop constructed in the proof of the lemma to $\gamma$.  The standard argument for the quasi-isometry invariance of filling functions, as in e.g.~\cite[Lemma 1]{Young}, then implies that $\gamma$ is hard to fill.

Each of the geodesics comprising $\gamma$ traverses any edge once.  Moreover, if the two geodesics share an edge, then it is an edge of $\widetilde{X_F}$ which they traverse with opposite orientation.  Therefore, viewed as an integral $1$-cycle, $\gamma$ has coefficient 1 on any edge in its support, and so it reduces to a mod $q$ cycle with the same mass.  Any integral filling reduces to a mod $q$ filling with smaller or equal mass.  Therefore, to prove the theorem, it suffices to show that $\gamma$ is hard to fill mod $q$: this also gives a lower bound on the size of an integral filling.

Now we proceed with the case $R=\mathbb{Z}/q\mathbb{Z}$.  The level set $L_G$ consists of copies of $L_{A}$ and $L_{Q}$, joined along copies of $L_{E}$.  Given one copy $L$ of $L_A$ or $L_Q$, we denote its (topological) boundary in $L_G$ by $\partial L$.  Clearly $\partial L$ is a union of copies of $L_E$.

Fix a copy $L_{0}$ of $L_{Q}$ in $L_G$, and embed $\gamma$ in $L_0$. Let $\tau$ be a simplicial chain that fills $\gamma$ in $L_G$.  Given a subcomplex $L$ of $L_G$, we write $L \cap \tau$ to denote the restriction of $\tau$ to $L$, that is, the chain whose coefficients coincide with those of $\tau$ on $L$ and are zero elsewhere.  Notice that since $L_E$ is one-dimensional, $\tau$ can be written as the sum of its restrictions to copies of $L_Q$ and $L_A$.

Our aim is to show that $\partial(L_0 \cap \tau)=\gamma$. It follows that 
\[\Dist_{F_m}(x) \preceq \Mass(L_0 \cap \tau) \leq \Mass \tau\]
as desired. To achieve this, we need the following assertion.  

\begin{lem} \label{Subtract}
If $\alpha \in C_{2}(L_{A}; \mathbb{Z}/q\mathbb{Z})$ has boundary supported in $\partial L_A$, then $\partial \alpha = 0$.
\end{lem}
Assume the lemma.  For any copy $L$ of $L_A$ in $L_{G}$, we apply the lemma to $L \cap \tau$; it follows that 
\[\partial(\tau-L \cap \tau)=\partial \tau,\]
and so we can replace $\tau$ with $\tau-L \cap \tau$.  Thus we reduce to the case that $\tau$ is supported on copies of $L_Q$.  But since these copies are disjoint, it follows that $\partial(L \cap \tau)=0$ for every copy $L$ of $L_Q$ other than $L_0$.  Therefore, we can also subtract off the restrictions to every other copy of $L_Q$ without changing the boundary.  Thus $\partial(L_0 \cap \tau)=\partial\tau=\gamma$.
\end{proof} 

\begin{proof}[Proof of Lemma \ref{Subtract}] We know that $L_{A}$ is homotopy equivalent to an infinite wedge sum of copies of $Y$,
\[Y_{\infty} = \bigvee_{j \in S_{A}} Y_{j},\]
where $S_{A}$ is the set of vertices in $\widetilde{X_{A}} \setminus L_{A}$. Notice that $\partial L_{A}$ consists of disjoint copies of $L_{E}$, which we denote $L_{E, i}$.  Each of these is homotopy equivalent to a wedge sum of scaled copies of squares indexed by disjoint subsets $S_{E, i} \subset S_A$, 
\[\Diamond_{\infty, i} =  \bigvee_{j \in S_{E, i}} \Diamond_{j}.\]
In fact, each $\Diamond_j$ is the square which generates the first homology group of the corresponding $Y_j$.  This defines an inclusion map $\bigsqcup_i \Diamond_{\infty, i} \rightarrow Y_{\infty}$ which corresponds, under the homotopy equivalences mentioned above, to the inclusion map $h: \partial L_A = \bigsqcup_i L_{E, i} \rightarrow L_{A}$.  Therefore, $h$ induces an injection on the level of homology mod $q$,
\[h_*: \bigoplus_i H_1(L_{E, i}; \mathbb{Z}/q\mathbb{Z}) \rightarrow H_1(L_{A}; \mathbb{Z}/q\mathbb{Z}).\]

By definition, $\partial \alpha$ is homologically trivial in $H_{1}(L_{A}; \mathbb{Z}/q\mathbb{Z})$.  This implies that $\partial \alpha$ restricted to each $L_{E, i}$ is homologically trivial in $H_{1}(L_{E,i}; \mathbb{Z}/q\mathbb{Z})$. Since $L_{E, i}$ is a $1$-complex, each such restriction is actually zero as a chain.  Thus $\partial\alpha=0$.
\end{proof}

\section{Higher dimensions}

\subsection{Construction and properties}
As before, we fix a prime $q$ and a $d \in \mathbb{N} \cup \{\infty\}$.
In the previous section, we built a $\CAT(0)$ group $G$ (depending on $q$ and $d$) as well as a kernel group $H \leq G$ such that $\FV^{2}_{H; R}$ grow at distinct rates for different $R$.  In higher dimensions, we inductively construct the groups $G_{n}$, $H_{n}$ starting from $G_1=G$.  Then Theorem \ref{B} provides us a way of proving that $\FV^{n + 1}_{H_{n}; R}$ grow at distinct rates for different $R$. 

Given $n \geq 2$, we define $G_n = G_{n - 1} \times F_{2}(u_{n}, v_{n})$ and 
$h_{n}: G_{n} \rightarrow \mathbb{Z}$ to be the group homomorphism sending $u_{n}$, $v_{n}$, and each generator of $G_{n - 1}$ to $1$. If we let $H_{n} = \ker h_{n}$, then by Proposition \ref{amalg}, $H_{n} \cong G_{n - 1} *_{H_{n - 1}} G_{n - 1}$. 

The groups $G_{n}$ and $H_{n}$ satisfy a few key properties. In particular, we show that $G_{n}$ is a $\CAT(0)$ group and $H_{n}$ is a group of type $\F^{n + 1}$. 

\begin{lem}
The group $G_{n}$ acts freely and geometrically on an $(n+2)$-dimensional $\mathrm{CAT}(0)$ cube complex $\widetilde{X_{G_n}}$.
\end{lem}
\begin{proof}
We showed in Lemma \ref{n=1:CAT(0)} that $G_1$ acts on a $3$-dimensional $\mathrm{CAT}(0)$ cube complex $\widetilde{X_G}$.  Since $G_n$ is the direct product of $G_1$ with $n-1$ copies of $F_2$, it acts on the universal cover $\widetilde{X_{G_n}}$ where
\[X_{G_{n}}=X_{G} \times (S^{1} \vee S^{1})^{n-1}. \qedhere\]
\end{proof}

\begin{lem} \label{lem:Fn}
For any group $G$ of type \emph{$\F^{n + 1}$} and subgroup $H$ of type \emph{$\F^{n}$}, $G *_{H} G$ is of type \emph{$\F^{n + 1}$}. 
\end{lem}
\begin{proof} 
Say a CW complex $X$ is \emph{$n$-aspherical} if its universal cover is $n$-connected. We would like to construct an $n$-aspherical $(n + 1)$-complex $Z$ with fundamental group $G *_{H} G$.

By assumption, there exists an $n$-aspherical $(n+1)$-complex $X_{G}$ with fundamental group $G$, and an $(n - 1)$-aspherical $n$-complex $X_{H}$ with fundamental group $H$.  Fix a cellular map $\phi:X_H \to X_G$ which induces the natural inclusion $H \hookrightarrow G$ on fundamental groups; such a map exists because $X_G$ is $n$-aspherical.  We build the complex $Z$ by taking a copy of $X_H \times [0,1]$ (with the product cell structure) and gluing it to copies of $X_G$ on either side via $\phi$.  This space has the desired fundamental group by the Seifert--van Kampen theorem.

It remains to show that $Z$ is $n$-aspherical.  Observe that $\widetilde{Z}$ consists of copies of $\widetilde{X_G}$, indexed by vertices of the Bass--Serre tree of the amalgamated free product $G *_H G$, glued along copies of $\widetilde{X_H} \times [0,1]$ which are indexed by edges of the Bass--Serre tree.  In particular, $\widetilde Z$ has an exhaustion by unions of pieces indexed by finite subtrees of the Bass--Serre tree.  To show that $Z$ is $n$-aspherical, it is enough to show that every such subspace is $n$-connected.

Let $U \subset \widetilde{Z}$ be the union of pieces indexed by a finite subtree $T$ of the Bass--Serre tree.  We show that $U$ is $n$-connected by induction on the size of $T$.  If $T$ has one vertex, then $U \cong \widetilde{X_G}$ and is $n$-connected by assumption.  Let $T'$ be a subtree attained by adding one vertex to $T$.  The corresponding subspace $U' \subset \widetilde{Z}$ can be decomposed up to homotopy equivalence as
\[U' \simeq U \cup_{\widetilde{X_H}} \widetilde{X_G}.\]
Since $U$ and $\widetilde{X_G}$ are $n$-connected and $\widetilde{X_H}$ is $(n-1)$-connected, the exactness of the Mayer--Vietoris sequence, together with the Hurewicz theorem, implies that $U'$ is $n$-connected as well.
\end{proof}

\begin{cor}
$H_{n}$ is of type \emph{$\F^{n + 1}$}. 
\end{cor}
\begin{proof}
We proceed by an induction on $n$. The case $n = 1$ was shown in Lemma \ref{FinPro}. 

Given $n \geq 2$ and $H_{n - 1}$ is of type \emph{$\F^{n}$}, we consider $H_{n} \cong G_{n - 1} *_{H_{n - 1}} G_{n - 1}$. Recall that $G_{n - 1}$ is a $\CAT(0)$ group of type \emph{$\F^{\infty}$}. Applying Lemma  \ref{lem:Fn} completes the inductive step. 
\end{proof}

\subsection{Bounds on homological filling functions}
We start by proving Theorem \ref{B}, which we restate below.
\begin{thm}\label{RestateB}
Let $n \geq 2$ and let $R$ be a quotient of $\mathbb{Z}$.  For any group $G$ of type $\F^{n + 1}$ and subgroup $H$ of type $\F^{n}$,
\[\FV^{n+1}_{G *_H G;R}(x) \preceq x\overline{\FV^{n+1}_{G;R}}\bigl(\overline{\FV^n_{H;R}}(x)\bigr),\]
where (as before) we write $\overline f(x)$ to mean $\max(f(x),x)$.  In addition, if $\FV^{n}_{G; R}(x) \preceq \exp(x)$, then
\[\FV^n_{H;R}\bigl({\overline{\FV^n_{G;R}}\,}^{-1}(x)\bigr) \preceq \FV^{n+1}_{G *_H G;R}(x).\]
\end{thm}
\begin{rmks}
\begin{enumerate}[(a)]
\item The decoration $\overline f$ specifies that the function must be at least linear.  Since sublinear functions are $\sim$-equivalent to linear ones, this simply chooses a ``reasonable'' representative of the equivalence class.
\item An analogous proof shows that
\[\FV^n_{H;\mathbb{Q}}\bigl(\overline{\FV^n_{G;\mathbb{Z}}}^{-1}(x)\bigr) \preceq \FV^{n+1}_{G *_H G;\mathbb{Q}}(x).\]
\end{enumerate}
\end{rmks}
\begin{proof}
Let $X_G$, $X_H$, and $Z=X_G \cup_\phi (X_H \times [0,1]) \cup_\phi X_G$ be complexes with fundamental groups $G$, $H$, and $G *_H G$, as in the proof of Lemma \ref{lem:Fn}. We denote the space $X_{H} \times \{1/2\}$ by $W \subset Z$. If $\widetilde{W}$ is the preimage of $W$ in $\widetilde{Z}$, then each component of $\widetilde{W}$ is homeomorphic to $\widetilde{X_{H}}$ and lies in the middle of a copy of $\widetilde{X_{H}} \times [0, 1]$.

\subsubsection*{Lower bound}
For the lower bound, it suffices to construct, for enough values of $y$, an $n$-cycle that has boundary mass $\preceq \overline{\FV^{n}_{G; R}}(y)$ and filling volume $\succeq \FV^{n}_{H; R}(y)$. Let $\alpha$ be a hard-to-fill $(n - 1)$-cycle in $\widetilde{X_{H}}$ with boundary mass $y$. The minimal filling $\omega$ of $\alpha$ in $\widetilde{X_{H}}$ has volume $\FV^{n}_{H; R}(y)$. Since $\widetilde{X_{H}} \times \{0\}$ embeds in $\widetilde{X_{G}}$, the minimal filling $\beta$ of $\tilde\phi_\#(\alpha \times \{0\})$ in $\widetilde{X_{G}}$ satisfies
\[\Mass^{n} (\beta) \preceq \FV^{n}_{G; R}(y).\]
The same holds for the minimal filling $\beta'$ of $\tilde\phi_\#(\alpha \times \{1\})$ in the corresponding other copy of $\widetilde{X_{G}}$. The fillings $\beta$ and $\beta'$, together with $\alpha \times [0, 1]$, form an $n$-cycle $\gamma$ with
\[\Mass^{n} (\gamma) \preceq \FV^{n}_{G; R}(y) + y \preceq \overline{\FV^{n}_{G; R}}(y).\]

We now argue that any filling $\eta \in C_{n+1}(\widetilde Z;R)$ of $\gamma$ must have $\Mass^{n + 1} (\eta) \succeq \FV^{n}_{H; R}(y)$.  To see this, notice that (for any consistent choice of transverse orientation) $\eta \cap \widetilde W$ is a cellular filling of $\alpha \times \{1/2\}$ in $\widetilde{X_H} \times \{1/2\}$.  Therefore
\[\FV^{n}_{H; R}(y) \preceq \Mass^n(\eta \cap \widetilde W) \leq \Mass^{n+1}(\eta).\]
This shows that $\FV^n_{H;R}(y) \preceq \FV^{n+1}_{G *_H G;R}(\FV^n_{G;R}(y))$, and therefore shows the lower bound on $\FV^{n+1}_{G *_H G;R}(x)$ for any $x$ in the image of $\FV^n_{G;R}$.

To conclude, we use the assumption that $\FV^n_{G;R}(y) \preceq \exp(y)$.  Then there is a constant $C>1$ such that for any sufficiently large $x$, there is some $y$ such that $x<\FV^n_{G;R}(y)<Cx$, so that
\[\FV^n_{H;R}\bigl({\overline{\FV^n_{G;R}}\,}^{-1}(x)\bigr) \preceq \FV^{n+1}_{G *_H G;R}(Cx).\]
In the absence of this assumption, we have proved that one function grows at least as fast as another when restricted to an infinite, but potentially sparse sequence of points.  This is insufficient to prove faster growth overall, even for functions that occur as asymptotic invariants in geometric group theory: see for example the groups of oscillating intermediate growth constructed in \cite{KaPak}.

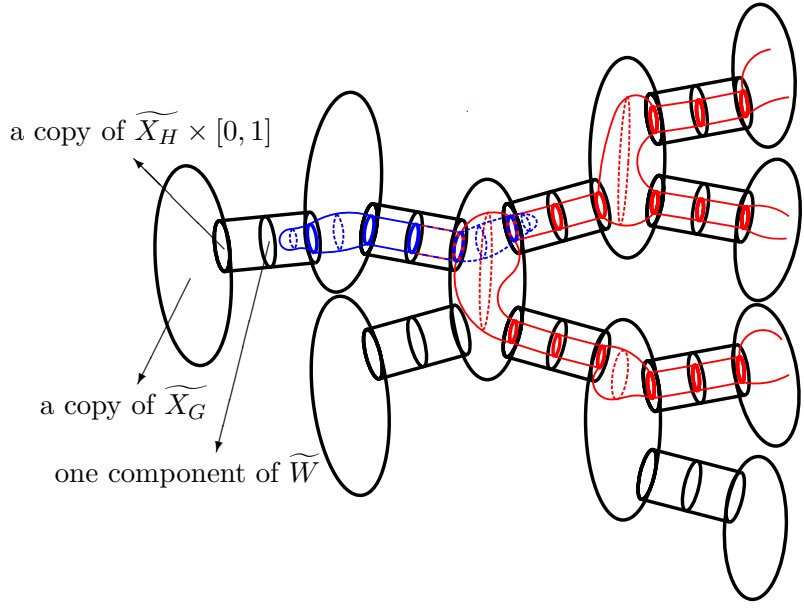
\begin{figure}[t]
\setlength{\unitlength}{1cm}
\centering
\begin{picture}(11,8)(0,0)
\put (2,0) {\includegraphics[scale = 0.5]{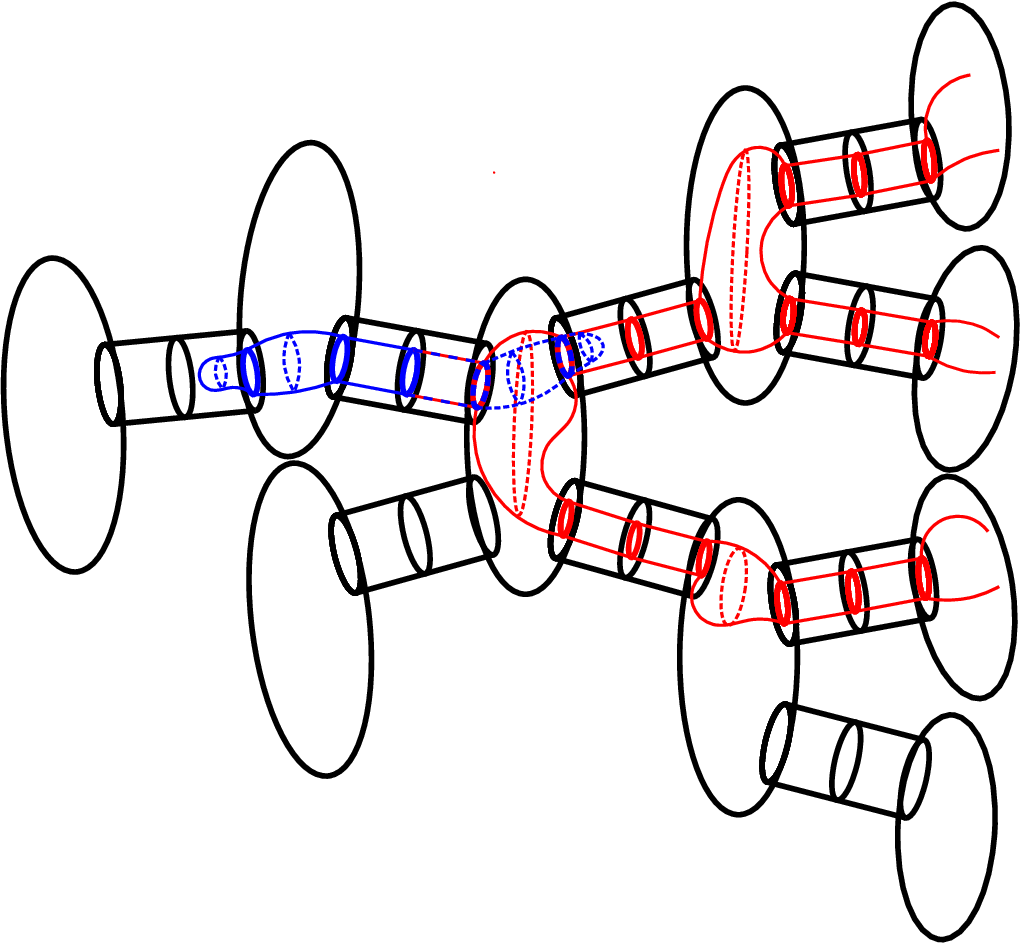}}
\put(3.55,4.8){\vector(-1,-4){0.7}}
\put(0.7, 1.6){one component of $\widetilde{W}$}
\put(2.95, 4.7){\vector(-1, 1){1.2}}
\put(0.1, 6.15){a copy of $\widetilde{X_{H}} \times [0,1]$} 
\put(2.5, 4.3){\vector(-1, -2){0.7}}
\put(0.5, 2.5){a copy of $\widetilde{X_{G}}$}
\end{picture}
\caption{A schematic illustration of the space $\widetilde Z$ and a typical cycle inside it.  We illustrate the process of reflecting a piece of the cycle across a component of $\widetilde{W}$.} \label{fig:aa}
\end{figure}

\subsubsection*{Upper bound}
Now let $Y_G=X_G \cup_\phi (X_H \times [0,1/2])$ (or the other half of $Z$, which is isometric, so we don't distinguish them in notation).  Observe that $\widetilde{Z}$ is covered by copies of $\widetilde{Y_G}$, and adjacent copies of $\widetilde{Y_G}$ always meet in a component of $\widetilde{W}$.

To handle the upper bound, we recall the following key facts:
\begin{enumerate}[(i)]
\item For any $(n - 1)$-cycle in one component of $\widetilde{W}$ with boundary mass $x$, the filling volume cannot exceed $\FV^n_{H;R}(x)$.
\item For any $n$-cycle in $\widetilde{Y_{G}}$ with boundary mass $x$, the filling volume cannot exceed $x + \FV^{n + 1}_{G;R}(x)$.
\end{enumerate}
Now, let $\gamma$ be a $n$-cycle in $\widetilde{Z}$ of mass $x$.  We will construct a filling of $\gamma$ of bounded size by inducting on the number of copies of $\widetilde{Y_G}$ that it intersects. Suppose first that it lives in two adjacent copies of $\widetilde{Y_{G}}$ connected by a component $W_0$ of $\widetilde{W}$, which slices $\gamma$ into two pieces $\gamma_1$ and $\gamma_2$.  These pieces overlap in an $(n-1)$-cycle
\[\partial\gamma_1=-\partial\gamma_2=\gamma \cap W_0 \in Z_{n-1}(W_0;R).\]

By (i), we can build a filling $\sigma \in C_n(W_0;R)$ of $\gamma \cap W_0$ with volume $\preceq \FV^{n}_{H;R}\bigl(\Mass^{n} (\gamma)\bigr)$.  Then, $\gamma_1-\sigma$ and $\gamma_2+\sigma$ are $n$-cycles in respective copies of $\widetilde{Y_G}$ whose sum is $\gamma$. By (ii),
\begin{equation} \label{2 pieces}
    \FVol_{G *_H G;R}(\gamma) \preceq 2\bigl(x + \FV^n_{H;R}(x)+\FV^{n + 1}_{G;R}[x + \FV^{n}_{H;R}(x)]\bigr)
    \preceq  2\overline{\FV^{n+1}_{G;R}}\bigl(\overline{\FV^n_{H;R}}(x)\bigr).
\end{equation}

Now in the general case, $\gamma$ is contained in the union of a finite subtree of copies of $\widetilde{Y_{G}}$ glued along components of $\widetilde{W}$. We can build an overall filling by inductively clipping off the pieces living in each copy. At each step, we choose a copy $Y_0$ of $\widetilde{Y_G}$ which is a leaf of the finite subtree; it is connected to the rest of the cycle through a component $W_0$ of $\widetilde W$.  We modify $\gamma$ by ``reflecting'' $\gamma \cap Y_0$ through $W_0$, that is, by subtracting off $\gamma \cap Y_0$ and adding the corresponding chain in the neighboring copy of $\widetilde{Y_G}$.  The resulting cycle $\gamma'$ is contained in the union of one less copy of $\widetilde{Y_G}$, and $\Mass^n(\gamma') \leq \Mass^n(\gamma)$.  Moreover, \eqref{2 pieces} gives a bound on the filling volume of $\gamma-\gamma'$.

The number of reduction steps is bounded above by $x$, so
\[\FVol(\gamma) \preceq x\overline{\FV^{n+1}_{G;R}}\bigl(\overline{\FV^n_{H;R}}(x)\bigr). \qedhere\] 
\end{proof} 

\begin{thm}\label{RestateA}
For $H_{1} = H$ and $H_{n} = G_{n - 1} *_{H_{n - 1}} G_{n - 1}$ as given above, we have 
\begin{align*}
        \FV^{n+1}_{H_{n};R}(x) &\preceq x^{(n + 1)[\ln(n + 1) + 2]} & R
        &= \mathbb{Q}\text{ or }\mathbb{Z}/p\mathbb{Z}, \gcd(p,q)=1\\
        \FV^{n+1}_{H_{n};R}(x) &\succeq f_{d,n}(x)
        & R &= \mathbb{Z}\text{ or }\mathbb{Z}/q\mathbb{Z},
    \end{align*}
    where $f_{d,n}(x)=\exp(\!\sqrt[n]{x})$ if $d=\infty$ and $x^{d/n}$ otherwise.
\end{thm}
\begin{proof}
First, we use induction on $n$ to obtain the lower bound for $\FV^{n + 1}_{H_{n}; R}$ when $R = \mathbb{Z}$ or $\mathbb{Z}/q\mathbb{Z}$.  The case $n=1$ is given by Theorem \ref{LowerBound}.  Moreover, $G_{n-1}$ is a CAT(0) group, and therefore $\FV^{n}_{G_{n - 1}; R} \preceq x^{\frac{n}{n - 1}}$.  Now suppose we have the lower bound for $n-1$; then Theorem \ref{RestateB} tells us that
\[f_{d,n-1}(x^{\frac{n-1}{n}}) \preceq \FV^{n + 1}_{H_{n};R}(x).\]
The function on the left is $f_{d,n}(x)$. 

Now we turn to the upper bound when $R = \mathbb{Q}$ or $\mathbb{Z}/p\mathbb{Z}$. By Proposition \ref{UpperBound}, we know that $\FV^{2}_{H_{1}; R}(x) \preceq x^{5}$. If $R = \mathbb{Z}/p\mathbb{Z}$, then we can inductively apply the right inequality in Theorem \ref{RestateB}, 
\[\FV^{n + 1}_{H_{n};\mathbb{Z}/p\mathbb{Z}}(x) \preceq x\overline{\FV^{n+1}_{G_{n - 1};\mathbb{Z}/p\mathbb{Z}}}\bigl(\overline{\FV^n_{H_{n - 1};\mathbb{Z}/p\mathbb{Z}}}(x)\bigr).
\]
Then the asymptotic upper bounds for $\FV^{n + 1}_{H_n; \mathbb{Z}/p\mathbb{Z}}$ form a sequence of functions $\{u_n(x)\}_{n = 1,2,\ldots}$ which satisfies $u_{1}(x) = x^{5}$ and the recurrence formula
\[u_{n}(x) = x[u_{n - 1}(x)]^{\frac{n + 1}{n }}.\]
A computation then yields
\[\FV^{n + 1}_{H_{n}; \mathbb{Z}/p\mathbb{Z}} \preceq u_{n}(x) \sim x^{(n + 1)[\frac{5}{2} + \sum_{k = 3}^{n + 1} \frac{1}{k}]} \preceq x^{(n + 1)[\ln(n + 1)+2]}.\]

If $R=\mathbb{Q}$, we are unable to straightforwardly use the result of Theorem \ref{RestateB}: the proof does not work since the rational filling of a $(n - 1)$-cycle in $\widetilde{W}$ will not necessarily give us an integral $n$-cycle in $\widetilde{Y_{G_{n-1}}}$.  Fortunately, according to Proposition \ref{UpperBound}, every integral cycle in $L_{G}$ has a rational filling with mass $\preceq x^{5}$ and integer multiples of $1/q$ as the coefficients.  Equivalently, every cycle with coefficients in $q\mathbb{Z}$ has an integral filling with mass $\preceq x^5$.  Generalizing this, we obtain: 

\begin{lem}
When $n \geq 2$,
\[\FV^{n+1}_{H_{n}; \mathbb{Q}}(x) \preceq \frac{x}{q}\bigl[x + u_{n-1}(x) + \bigl(x + q \cdot u_{n-1}(x)\bigr)^{\frac{n + 1}{n}}\bigr] \preceq u_n(x),\]
and moreover we can always find a filling satisfying this bound whose coefficients are in $(1/q)\cdot\mathbb{Z}$.
\end{lem}
\begin{proof}
  We prove this by induction on $n$.  Using the lemma in dimension $n-1$, we can find efficient integral fillings in copies of $\widetilde{X_{H_{n-1}}}$ for cycles with coefficients in $q\mathbb{Z}$.  Given an integral $n$-cycle $\gamma$ in $H_n$, we apply the proof of the upper bound in Theorem \ref{RestateB} to $q \cdot \gamma$: iteratively clip off the outermost portions of $q \cdot \gamma$.  At each step, we end up with a new cycle $\gamma'$ with coefficients in $\mathbb{Z}$, such that $\Mass^n(\gamma') \leq \Mass^n(\gamma)$, and $q(\gamma-\gamma')=\partial \eta$ for some integral chain $\eta$ which satisfies
  \[\Mass^{n+1}(\eta) \preceq x + u_{n-1}(x) + \FV^{n + 1}_{G_{n - 1}; \mathbb{Q}}\bigl(x + q \cdot u_{n-1}(x)\bigr).\]
  After at most $x$ steps we end up with a cycle contained in two copies of $\widetilde{X_{G_{n-1}}}$, which we can fill in a similar way.  The sum of the various $\eta$ produced at each step is an integral $(n+1)$-chain which fills $q \cdot \gamma$ and satisfies the desired bound on mass.
\end{proof}

Applying the recurrence in the lemma, we once again get that $\FV^{n + 1}_{H_{n}; \mathbb{Q}}(x) \preceq x^{(n + 1)[\ln(n + 1) + 2]}$.
\end{proof}

\bibliographystyle{amsplain}
\bibliography{filling}
\end{document}